\RequirePackage[l2tabu, orthodox]{nag}
\documentclass[a4paper]{scrartcl}
\usepackage{typearea}
\typearea{12}

\usepackage{amssymb,amsthm,amsmath,braket,amsfonts,amsopn}
\usepackage{authblk, enumitem}
\usepackage[ruled, lined]{algorithm2e}
\usepackage{ulem}

\usepackage{enumitem}

\newtheorem{theorem}{Theorem}[section]
\newtheorem{proposition}[theorem]{Proposition}
\newtheorem{lemma}[theorem]{Lemma}

\newtheorem{corollary}[theorem]{Corollary}
\newtheorem{remark}[theorem]{Remark}
\newtheorem{example}[theorem]{Example}

\newtheorem{condition}{Condition}

\renewcommand{\Bbb}[1]{\mathbb{#1}}
\newcommand{\R}{\Bbb{R}}
\renewcommand{\S}{\Bbb{S}}
\newcommand{\B}{\Bbb{B}}
\newcommand{\cU}{\mathcal{U}}
\newcommand{\cV}{\mathcal{V}}

\newcommand{\tX}{\widetilde{X}}
\newcommand{\tZ}{\widetilde{Z}}
\newcommand{\ty}{\tilde{y}}

\newcommand{\hU}{\widehat{U}}
\newcommand{\hV}{\widehat{V}}

\newcommand{\hy}{\hat{y}}

\newcommand{\rank}{r}
\renewcommand{\Im}{\mathop\mathrm{Im}}
\renewcommand{\vec}{\mathop\mathrm{vec}}

\newcommand{\rint}{\mathop\mathrm{rint}}

\newcommand{\Span}{\mathop\mathrm{Span}}

\usepackage{amsopn}
\DeclareMathOperator{\He}{He}

\newenvironment{keywords}%
   {\begin{trivlist}\item[]{\bfseries Keywords:}\ }
   {\end{trivlist}}
\newenvironment{classification}%
   {\begin{trivlist}\item[]{2010 Mathematical subject classification:}\ }
   {\end{trivlist}}

\begin{document}

\title
{Perturbation Analysis of Singular Semidefinite Programs and Its
Applications to Control Problems%
}



\author[1]{Yoshiyuki Sekiguchi\thanks{2-1-6, Etchujima, Koto, Tokyo 135-8533, JAPAN. yoshi-s@kaiyodai.ac.jp}}
\author[2]{Hayato Waki\thanks{744 Motooka, Nishi-ku, Fukuoka 819-0395, JAPAN. waki@imi.kyushu-u.ac.jp}}
\affil[1]{Tokyo University of Marine Science and Technology}
\affil[2]{Institute of Mathematics for Industry, Kyushu University}

\date{}



\maketitle

\begin{abstract}
We consider sensitivity of a semidefinite program  under
 perturbations
 in the case that the primal problem is strictly feasible and the dual problem is weakly feasible.
 When the coefficient matrices are perturbed,
 the optimal values can change discontinuously
 as explained in concrete examples.
 We show that the
 optimal value of such a semidefinite program changes continuously
under conditions involving the behavior of the minimal faces of the perturbed dual problems.
 In addition, we determine what kinds of perturbations keep the minimal
 faces invariant, by using the reducing certificates, which are produced in facial reduction. Our results allow us to classify the behavior of the minimal face of a semidefinite program obtained from a control problem. 
\end{abstract}
%

\section{Introduction}
\label{section:intro}

A \textit{semidefinite program} is the problem of
maximizing
a linear function subject to the constraint
that an affine combination of matrices is positive semidefinite, where the constraint is called a linear matrix inequality.
Semidefinite programs have various applications, such as discrete optimization,
polynomial optimization and control problems (e.g. 
\cite{Anjos12,Scherer06}
). 
If the feasible sets of a semidefinite program and the dual problem satisfy the
constraint qualifications, both of which are called \textit{strict feasibility}, then
interior point methods compute an approximation to
an exact solution efficiently; see, e.g., \cite{deKlerk02,Tuncel10}.

With a lack of strict feasibility, 
interior point methods are numerically unstable and
often give wrong optimal values [5-7].
To avoid such numerical instability,
we can use the technique called \textit{facial reduction},
which finds the minimal face among
the faces of the positive semidefinite cone containing a feasible set.
Such a face is called the \textit{minimal face of a semidefinite
program} [8-11].
Then, we obtain a semidefinite program that satisfies
strict feasibility and has the same optimal value
as the original problem. 
For applications of facial reduction,
see the monograph \cite{Drusvyatskiy17} and the references therein. 

The first contribution of this paper is to provide sufficient conditions
for continuity of the optimal value under perturbations, in the case that the primal problem
is strictly feasible and the dual problem is feasible but not strictly
feasible (Theorem \ref{thm:main2}).
In that case, if we perturb the constant matrix
in the constraint of the primal problem,
then it can be shown from the general theory of convex analysis
that the optimal value changes continuously \cite[Corollary 7.5.1]{Rockafellar70}.
For more detailed analysis, see \cite{Cheung14}.
However if we also perturb the coefficient matrices of the variables, then
the optimal value may change discontinuously (Example \ref{ex:notrank}, \ref{ex:notface}).
Here one of the keys to the phenomenon is the behavior of the minimal face of the dual
problem under the perturbation.
By using concrete examples, we argue that our sufficient conditions are hard to remove.

In the case that both of the primal and dual problems are strictly feasible,
continuity of the optimal value can be shown by Gol'{\v s}hte{\u \i}n \cite[Theorem 17]{Golshtein72}.
Moreover if perturbations are restricted on the constant matrices in the constraint of a semidefinite program, several authors have studied
stability of optimal solutions; see, e.g., [16-18]. 
%
%
%
%
%
Perturbation analysis of general nonlinear programming has been studied thoroughly by Bonnans and Shapiro \cite{Bonnans00}.

The second contribution is to obtain sufficient conditions
for the perturbations to keep the minimal face invariant (Proposition
\ref{prop:invariant1}, \ref{prop:invariant3}). 
If the minimal face does not change under a perturbation, then one of the conditions in Theorem \ref{thm:main2} is satisfied. 
These results give a new insight to perturbation analysis of semidefinite programs.
Here we use \textit{reducing certificates}, which
are generated by facial reduction to find the minimal face \cite{Pataki13}. 
We remark that reducing certificates are often obtained without solving
semidefinite programs if the problems are generated from combinatorial optimization problems,
matrix completion problems, sums of squares problems \cite{Drusvyatskiy17} or $H_\infty$
control problems \cite{Waki15}.
Using these conditions, we investigate a semidefinite program generated from an $H_\infty$ state feedback
control problem.

The organization of this paper is as follows: preliminaries on
semidefinite programs and facial reduction  are given in Section \ref{preliminary}. In Section
\ref{sec:main}, we show the main result on continuity of
the optimal value of a semidefinite program.
In Section \ref{section:face_invariant}, we give sufficient conditions on
the perturbations under which the minimal
face does not change. We devote Section \ref{sec:example} to
applications to a control problem and numerical experiements. The conclusions are given in Section \ref{sec:conclusion}.

\section{Preliminaries on 
 Semidefinite Program and Facial Reduction}\label{preliminary}
 \subsection{\textbf{Semidefinite Program}}
 Let $\mathbb{S}^n$, $\mathbb{S}^n_+$ and $\mathbb{S}^n_{++}$ be the sets of $n\times n$ symmetric matrices, positive semidefinite matrices and  positive definite matrices, respectively.  In this paper, the primal semidefinite program (SDP) \eqref{P} and its dual \eqref{D} are   formulated as follows: 
\begin{align}
\label{P}
&\quad  \sup_{y, Z}\left\{ b^Ty  : A_0 - \sum_{k=1}^m y_k A_k =Z, \ y\in \R^m,\ Z \in\Bbb{S}^n_+\right\}, \tag{$P$}\\
 \label{D}
&\quad\inf_{X}\left\{A_0 \bullet X :  A_k \bullet X = b_k \ (k \in [m]), X \in\Bbb{S}^n_+\right\}, \tag{$D$}
\end{align}
where $A_0, A_1, \ldots, A_m\in\mathbb{S}^n$,  $b\in\Bbb{R}^m$, $[m] := \{1,
\ldots, m\}$, and the inner product $A\bullet B$ is defined by $\sum_{i, j=1}^n A_{ij}B_{ij}$ for $A, B\in\mathbb{S}^n$. 

Problem \eqref{P} is said to be {\itshape strictly feasible} if there exists a feasible solution $(y, Z)$ in \eqref{P} such that $Z\in\mathbb{S}^n_{++}$. Problem \eqref{D} is said to be {\itshape strictly feasible} if there exists a feasible solution $X$ in \eqref{D} such that $X\in\mathbb{S}^n_{++}$.  We say  that \eqref{P} (resp. \eqref{D}) is {\itshape weakly feasible}, if \eqref{P} (resp. \eqref{D}) is feasible but not strictly feasible.

Throughout this paper, we deal with only the case where both \eqref{P} and its dual \eqref{D} are feasible. 
We say that \eqref{P} is {\itshape nonsingular} if both \eqref{P} and \eqref{D} are strictly feasible and the coefficient matrices $A_1, \ldots, A_m$ are linearly independent. We say that \eqref{P} is {\itshape singular} if the coefficient matrices are linear dependent or at least one of \eqref{P} and \eqref{D} is weakly feasible. 

\subsection{\textbf{Facial Reduction for SDP}}
The definition of a face of a general convex set is provided in \cite{Rockafellar70}. The following lemma provides results on a facial structure of $\Bbb{S}^n_+$, e.g. \cite{Borwein81b,Pataki13}. 

  \begin{lemma}\label{facest}
\begin{enumerate}
\item\label{F1} Any face of $\Bbb{S}^n_+$ is either the
     empty set, $\{O_{n\times n}\}$,
     $\Bbb{S}^n_+$, or        
\begin{align*}
&\left\{
Q\begin{pmatrix}
O_{(n-r)\times (n-r)} & O_{(n-r)\times r}\\
O_{r\times (n-r)} & M
\end{pmatrix}Q^T : M\in\Bbb{S}^r_+
\right\}, 
\end{align*}
where $Q$ is an $n\times n$ nonsingular matrix. 
     
\item\label{F2} The set $\Bbb{S}^n_+ + F^{\perp}$
is closed for all faces $F$ of $\Bbb{S}^n_+$, where $F^{\perp}$ stands
     for the set $\{Z\in\Bbb{S}^n : Z\bullet X = 0 \ (\forall X\in
     F)\}$. 
\end{enumerate}
\end{lemma}
We call $Q$ in Part \ref{F1} of Lemma \ref{facest} the \textit{nonsingular matrix associated to the face}.
     It follows from this property that for any $U\in\mathbb{S}^n_+$, 
the set $\Bbb{S}^n_+\cap\{U\}^{\perp}$ is a face of $\Bbb{S}^n_+$,
     where 
     $\{U\}^{\perp} = \{ X\in\Bbb{S}^n : X\bullet U = 0\}$. 
The property given in Part \ref{F2} of Lemma \ref{facest}, which is called niceness, 
     implies that $F^* = \Bbb{S}^n_+ + F^{\perp}$ for all faces $F$ of
     $\Bbb{S}^n_+$. Here $F^*$ is the dual cone of $F$,  i.e.  $F^*=\{Z\in\Bbb{S}^n : Z\bullet X \ge 0 \ (\forall X\in F)\}$.

We define the minimal face of \eqref{D} and introduce facial reduction for \eqref{D}. 
The minimal face of \eqref{D} is defined as the intersection of all faces of $\mathbb{S}^n_{+}$ that contain the feasible region of \eqref{D}. We denote the minimal face by $F_{\min}$. The following result on the minimal face is obtained by \cite{Pataki13} and Part \ref{F2} in Lemma \ref{facest}.

 \begin{lemma}
  \cite[\textrm{SDP version of Section 28.2.6 and Lemma 28.4}]{Pataki13}\label{freduction}
Assume that \eqref{P} and \eqref{D} are feasible. Let $F$ be a face
 of $\Bbb{S}^n_+$ that contains $F_{\min}$ and $\rint F$ be its
 relative interior. Then
 the following are equivalent;
\begin{enumerate}
 \item $F\neq F_{\min};$
 \item
      \label{lemma:freduction_cond}
 There exists $(y, U, V)\in\Bbb{R}^m\times \Bbb{S}^n_+\times
      F^{\perp}$ such that
 \begin{align}
 \label{certificate}
b^Ty &= 0,\quad -\sum_{k\in [m]}y_kA_k = U+V \mbox{ and } U+V \not\in F^{\perp}; 
 \end{align}
 \item $\displaystyle \left\{X\in\rint F : A_k\bullet X = b_k \
       (k\in [m]) \right\}=\emptyset.$
\end{enumerate}
  If $U$ satisfies the system in \ref{lemma:freduction_cond},  
  then we have $F_{\min}\subseteq F\cap \{U\}^{\perp}\subsetneq F$. 
 \end{lemma}
We call the above system \eqref{certificate} the \textit{discriminant system} of the facial reduction
for \eqref{D}, and  a solution $(y, U, V)$ a \textit{reducing certificate}. 

The facial reduction for SDP in e.g. \cite{Pataki13,Waki13}
is a procedure based on Lemma \ref{freduction}.
 It generates a sequence
 $\{F_i\}_{i=0}^s$ of  faces of $\Bbb{S}^n_+$
 such that 
 \begin{align*}
    F_0 = \S^n_+,\ F_{i} &= F_{i-1}\cap\{U^i\}^\perp \ (i=1, \ldots, s) \mbox{ and } F_s =F_{\min}. 
    \end{align*}
Therefore, the iterative process can be expressed as 
    \begin{align*}
     \S_+^n &= F_0 \overset{(y^1, U^1, V^1)}{\longrightarrow} F_1
   \overset{(y^2, U^2, V^2)}{\longrightarrow} F_2
   \overset{(y^3, U^3, V^3)}{\longrightarrow}
   \cdots
   \overset{(y^s, U^s, V^s)}{\longrightarrow} F_s=F_{\min},
 \end{align*}
 where we call $\{(y^i, U^i, V^i)\}_{i=1}^s$ a \textit{facial reduction sequence} for \eqref{D}.
 Here we note that $U^i,V^i$ need to satisfy $U^i+V^i \notin F_{i-1}^\perp$. 
 Examples of facial reduction for SDP can be seen in e.g. \cite[Example 28.3]{Pataki13} and \cite[Example 3.1]{Waki13}. 

If the discriminant system \eqref{certificate} has multiple solutions, then we have flexibility in choosing a facial reduction
sequence for \eqref{D}. Cheung and Wolkowicz \cite[Proposition
B.1]{Cheung14} prove that any two facial reduction sequences must
be of the same length 
when a reducing certificate $(y,U,V)$ is selected at each iteration
so that $U$ has the maximal rank. The length is called the
degree of singularity for \eqref{D}. 
%
The degree of singularity is used in \cite{Cheung14} for the sensitivity analysis of SDPs and in \cite{Sturm00b} for the error bounds.

Although we deal with only the feasible SDPs in the present paper, we introduce a study on the infeasibility briefly. Infeasibility of SDP has two types as well as feasibility, i.e. strong infeasibility and weak infeasibility.  The authors in \cite{Liu18,Lourenco16} discuss a characterization of  infeasibility by facial reduction.

\section{Main Result}\label{sec:main}
\subsection{\textbf{Stability of Singular Semidefinite Programs}}\label{SlaterFailsD}
We define the perturbed problems for \eqref{P} by 
\begin{align}
\label{Pt}  
& \sup_{y, Z}\left\{ b(t)^Ty  : \sum_{k \in [m]} y_k A_k(t) + Z = A_0(t),\
 y\in \R^m,\ Z \in\Bbb{S}^n_+\right\}, \tag{$P_t$}\\
\label{Dt}  
&\inf_{X}\left\{A_0(t) \bullet X :  A_k(t) \bullet X = b_k(t) \ (k \in [m]),\ X \in\Bbb{S}^n_+\right\}, \tag{\text{$D_t$}}
 \end{align}
where $t\ge 0$, $A_k(t)\in \S^n,\ b(t)\in \R^m$ are continuous at $t=0$, and $A_k(0) = A_k,\ b(0) = b$. 

In this subsection, the following conditions are imposed on the initial SDP:
    \begin{condition}
     \label{cond:singular}
   \hspace{1ex}
  \begin{enumerate}[label=(C\arabic*)]
   \item\label{C1} \eqref{D} is feasible, and \eqref{P} is strictly feasible;
   \item\label{C2} $A_1, \ldots, A_m$ are linearly independent.
  \end{enumerate}
    \end{condition}
Then, by applying the  facial reduction to \eqref{D}, there exist a nonsingular
matrix $Q$ and $r\in \Bbb{N}$ such that
{\small %
\begin{align}
 &\inf_{X_3}\left\{Q^TA_0Q \bullet  \begin{pmatrix}
  O & O \\
  O & X_3
 \end{pmatrix} : 
 \begin{array}{l}
 Q^TA_kQ \bullet
 \begin{pmatrix}
  O & O \\
  O & X_3
 \end{pmatrix} = b_k \ (k \in [m]),\ 
   X_3 \in \Bbb{S}^r_+
   \end{array}
 \right\} \label{FD}\tag{\text{$F(D)_0$}} 
\end{align}
}%
 has the same optimal value as \eqref{D}, and \ref{FD} is strictly feasible due to Lemma \ref{freduction}.
 Here, for $n\times n$ matrix $M$, we denote by $M_3$ the right bottom
 block of the partitioning 
\begin{equation}
   M = \begin{pmatrix}
       M_1 & M_2^T \\
       M_2 & M_3
       \end{pmatrix},
       \label{eq:partition}
\end{equation} 
where the partitioning is uniquely determined by Lemma \ref{facest} for the minimal face of \eqref{D} with $M_1 \in \S^{n-r}, M_2 \in \R^{r\times (n-r)}, M_3\in \S^r$.
We call $M_3$ the \textit{third block} of $M$ associated to the minimal face of \eqref{D}.
Then we can rewrite \ref{FD} as follows:
\begin{align}
\label{FD2}
 &\inf_{X}\left\{ (Q^TA_0Q)_3 \bullet X :  
 (Q^TA_kQ)_3 \bullet X = b_k \ (k\in [m]),  X \in\Bbb{S}^r_+ 
 \right\}. \tag{$F(D)$}
\end{align}
For $A=(a_{ij})_{1\le i, j\le n}\in\Bbb{S}^n$,
we define $\vec(A)$ as the vectorization of $A$, i.e., 
\[
\vec(A) = (a_{11}, a_{12}, \ldots, a_{1n}, a_{21}, a_{22}, \ldots, a_{n1}, \ldots, a_{nn})^T. 
\]
Let $\rank(A_1, \ldots, A_m)$ be the rank of the  matrix $(\vec(A_1),  \ldots, \vec(A_m))$.

The following theorem is the main result of this
paper. 
 \begin{theorem}
  \label{thm:main2}
  Under Condition \ref{cond:singular}, suppose that the minimal face $F_{\min}$ of \eqref{D} can be written as
 \[
 F_{\min} = \left\{
 Q
 \begin{pmatrix}
  O_{(n-r)\times(n-r)} & O_{(n-r)\times r} \\
  O_{r\times (n-r)} & X
 \end{pmatrix}Q^T : X \in\Bbb{S}^r_+\right\}
 \]
 for some nonsingular matrix $Q\in\Bbb{R}^{n\times n}$ and $r\in\Bbb{N}$.
 In addition, we suppose that the set $\{(A_0(t), \ldots, A_m(t), b(t)) : 0 \leq t \leq \delta\}$ satisfies the following assumptions for some $\delta>0$:
\begin{enumerate}
 \item\label{thm:sing:cond1} \eqref{Dt} is feasible for each $t\in [0,\delta]$;
 \item\label{thm:sing:cond2} For each $t\in [0,\delta]$,
      there exists a nonsingular matrix $Q(t)$ such that
       $\displaystyle\lim_{t\to 0}Q(t) = Q$, and the minimal face of \eqref{Dt} can
       be written as
       \[
       \left\{
       Q(t)
       \begin{pmatrix}
  O_{(n-r)\times(n-r)} & O_{(n-r)\times r} \\
  O_{r\times (n-r)} & X
 \end{pmatrix}Q(t)^T : X \in\Bbb{S}^r_+\right\};
	\]
 \item\label{thm:sing:cond3} For each $t\in [0,\delta]$, we have 
\[
      \rank \left((Q(t)^TA_1(t)Q(t))_3, \ldots,
       (Q(t)^TA_m(t)Q(t))_3 \right) \\
       = \rank \left( (Q^TA_1Q)_3, \ldots, (Q^TA_mQ)_3
 \right),
\]
where $M_3$ is the third block of $M\in \Bbb{S}^n$ associated with 
      the minimal face of \eqref{D}.
\end{enumerate}
Then the optimal value of \eqref{Dt} varies continuously at $t = 0$.
 \end{theorem}
 %
The following is an immediate corollary.
 \begin{corollary}
  \label{cor:cor1}
 Under Condition \ref{cond:singular}, suppose that there exists $\delta>0$ such that 
 \eqref{Dt}  has a nonempty feasible set
 and the same minimal face as \eqref{D}, and
 \[
\rank \left( (Q^TA_1(t)Q)_3, \ldots,
       (Q^TA_m(t)Q)_3 \right)
       = \rank \left( (Q^TA_1Q)_3, \ldots, (Q^TA_mQ)_3
 \right)
 \]
 for $t\in [0,\delta]$. Then the optimal value of \eqref{Dt} varies continuously at $t = 0$.
 \end{corollary}


Before proceeding to the proof, we investigate examples and
show that the rank condition
or the condition on the face can not be removed from Theorem
\ref{thm:main2} and Corollary \ref{cor:cor1}.

	  \begin{example}
	   \label{ex:notrank}
	  The following example satisfies the condition on the face
	   but does not satisfy the rank condition.
	 We set $b = (0, 2, 2)^T$ and
\[
   A_0 = \left(
 \begin{smallmatrix}
      0 & 0 & 0 & 0\\
	  0 & 0 & 0 & 0\\
  	  0 & 0 & 0 & 0\\
	  0 & 0 & 0 & 1
 \end{smallmatrix}
 \right),\ 
 A_1 = \left(\begin{smallmatrix}
                    1 & 0 & 0 & 0\\
                    0 & 1 & 0 & 0\\
                    0 & 0 & 0 & 0\\
                    0 & 0 & 0 & 0
                    \end{smallmatrix}\right),\ 
  A_2 = \left(\begin{smallmatrix}
  0 & 0 & 0 & 0\\
	0 &  0 & 0 & 0\\
  	0 &  0 & 1 & 0\\
	0 &  0 & 0 & 1
	      \end{smallmatrix}
 \right),
 A_3 = \left(
 \begin{smallmatrix}
    0 & 1 & 0 & 0\\
	1 & 0 & 0 & 0\\
  	0 & 0 & 1 & 0\\
	0 & 0 & 0 & 1
 \end{smallmatrix}
 \right)
\]	 
 in \eqref{P} and \eqref{D}. Then $A_1, A_2, A_3$ are linearly
   independent, $(P)$ is strictly feasible, and $(D)$ is weakly
	 feasible. The optimal value is $0$
	 and an optimal pair is
	 $X = \left(\begin{smallmatrix}
	         0 & 0 & 0 & 0\\
		     0 & 0 & 0 & 0 \\
		     0 & 0 & 2 & 0 \\
		     0 & 0 & 0 & 0 
	 \end{smallmatrix}\right),\ y = (0, 0, 0),\
	 Z = \left(\begin{smallmatrix}
		     0 & 0 & 0 & 0 \\
		     0 & 0 & 0 & 0 \\
		     0 & 0 & 0 & 0 \\
		     0 & 0 & 0 & 1 
	   \end{smallmatrix}\right)$.
The minimal face of \eqref{D} is
\[F_{\min} =\left\{\left(\begin{smallmatrix}
				     O_{2\times 2} & O_{2\times 2} \\
				     O_{2\times 2} & X_3
				    \end{smallmatrix}\right)\in \S^4_+: 
 X_3 \in
\S^2_+ \right\}.\]	 
	 If we perturb the matrices as
\[
   A_i(t) = A_i\ (i=0,1, 2), A_3(t) =
 \left(\begin{smallmatrix}
      0 & 1 & 0 & 0 \\
	  1 & 0 & 0 & 0 \\
  	  0 & 0 & 1 + t & 0\\
	  0 & 0 & 0 & 1 - t
 \end{smallmatrix}\right),
\]
	 then \eqref{Dt} remains feasible for each $t>0$.
	 In fact the feasible points of \eqref{Dt} can be
	 written as $X = \left(\begin{smallmatrix}
	            0 & 0 & 0 & 0 \\
				0 & 0 & 0 & 0 \\
				0 & 0 & 1 & \alpha \\
				0 & 0 & \alpha & 1 
			       \end{smallmatrix}\right)\
	   (-1\leq\alpha\leq 1)$.
	 Thus the minimal face of \eqref{Dt} is equal to $F_{\min}$ for
	   each $t>0$.
	 Now the dimension of the span of the third blocks of
	 the matrices $A_1, A_2, A_3$ is $1$, while that of $A_1(t),
	 A_2(t), A_3(t)$ is $2$ for each $t>0$. 
	   The optimal value of \eqref{Dt} is $1$
	   and the optimal pairs are
	   $X = \left(\begin{smallmatrix}
	            0 & 0 & 0 & 0 \\
				0 & 0 & 0 & 0 \\
				0 & 0 & 1 & \alpha \\
				0 & 0 & \alpha& 1 
			       \end{smallmatrix}\right)\
	   (-1\leq\alpha\leq 1),\
	   y = \left(\beta, \frac{1+t}{2t}, -\frac{1}{2t}\right),\
	   Z = \left(\begin{smallmatrix}
	          -\beta & \frac{1}{2t} & 0 & 0 \\
		      \frac{1}{2t} & -\beta & 0 & 0 \\
		      0 & 0 & 0 & 0\\
		      0 & 0 & 0 & 0
	   \end{smallmatrix}\right)\ \left(\beta\leq -\frac{1}{2t}\right)$ for each $t>0$.
	   Thus the optimal value changes discontinuously at $t=0$.
	  \end{example}

 \begin{example}
  \label{ex:notface}
  The following example satisfies the rank condition but
  does not satisfy the condition on the face.
  We set $b = (2, 2, 2, 0)^T$ and
\[
   A_0 = \left(
 \begin{smallmatrix}
	  0 & 0 & 0\\
  	  0 & 0 & 0\\
	  0 & 0 & 1
 \end{smallmatrix}
 \right),
 A_1 = \left(
 \begin{smallmatrix}
	  0 & 0 & 0\\
  	  0 & 1 & 0\\
	  0 & 0 & 0
 \end{smallmatrix}
 \right),
  A_2 = \left(\begin{smallmatrix}
	  0 & 0 & 0\\
  	  0 & 0 & 1\\
	  0 & 1 & 0
	      \end{smallmatrix}
 \right),
 A_3 = \left(
 \begin{smallmatrix}
	  1 & 0 & 0\\
  	  0 & 0 & 1\\
	  0 & 1 & 0
 \end{smallmatrix}
 \right),  
 A_4 = \left(
 \begin{smallmatrix}
	  0 & 0 & 1\\
  	  0 & 0 & 0\\
	  1 & 0 & 0
 \end{smallmatrix}
 \right) 
\]
 in \eqref{P} and \eqref{D}.
 Then $A_1, \ldots, A_4$ are linearly independent, \eqref{P} is strictly
 feasible, and  
 \eqref{D} is weakly feasible.
 The optimal value is $\frac{1}{2}$, and
 the optimal pairs are
 $X=\left(
 \begin{smallmatrix}
  0 & 0 & 0 \\
  0 & 2 & 1 \\
  0 & 1 & \frac{1}{2}
  \end{smallmatrix}\right),\
  y=\left(-\frac{1}{4}, \frac{1}{2}-y_3,y_3,0\right),
 Z=\left(
 \begin{smallmatrix}
  -y_3 & 0 & 0\\
  0 & \frac{1}{4} & -\frac{1}{2} \\
  0 & -\frac{1}{2} & 1
  \end{smallmatrix}\right),\ y_3\leq0$.
 The minimal face of \eqref{D} is
\[F_{\min} =\left\{\left(\begin{smallmatrix}
				     0 & O_{1\times 2} \\
				     O_{2\times 1} & X_3
				    \end{smallmatrix}\right)\in \S^3_+: 
 X_3 \in
\S^2_+ \right\}.\]  
 If we perturb the matrices as
 \begin{align*}
 & A_i(t) = A_i\ (i=0, 1,2), A_3(t) =
 \left(\begin{smallmatrix}
	  1 & 0 & 0\\
  	  0 & 0 & 1-t^2\\
	  0 & 1-t^2 & 0	
 \end{smallmatrix}\right),
 A_4(t) =
 \left(\begin{smallmatrix}
	  0 & 0 & 1\\
  	  0 & -2t & 0\\
	  1 & 0 & 0	
 \end{smallmatrix}\right), 
 \end{align*}
 then \eqref{Dt} is strictly feasible for each $t>0$. 
  In fact,
   $X=\left(
 \begin{smallmatrix}
	  2t^2 & 0 & 2t\\
  	  0  & 2 & 1\\
	  2t & 1 & 3
  \end{smallmatrix}\right)$ are strict feasible points
  of \eqref{Dt}.
  Thus the minimal face of \eqref{Dt} is $\mathbb{S}^3_+$ for each $t>0$.
 Since the span of the third blocks of the matrices $A_1(t),\ldots, A_4(t)$ has the same basis
 as that of $A_1,\ldots, A_4 $ for each $t>0$, the rank condition is satisfied.
 However the optimal value of \eqref{Dt} is $2$ with
$X=\left(
 \begin{smallmatrix}
	  2t^2 & t & 2t\\
  	  t & 2 & 1\\
  2t & 1 & 2
 \end{smallmatrix}\right),\ 
 y=\left(2, \frac{1}{t^2} - 1,-\frac{1}{t^2},\frac{1}{t}\right),
 Z=\left(
 \begin{smallmatrix}
	  \frac{1}{t^2} & 0 & -\frac{1}{t}\\
  	  0  & 0 & 0\\
	  -\frac{1}{t} & 0 & 1
 	  \end{smallmatrix}\right)
$
 being the unique optimal pair for each $t>0$. Thus the optimal value changes discontinuously at $t=0$.


 \end{example}

 \begin{example}\label{ex:theorem}
  Consider the same SDP as in Example \ref{ex:notface}.
   If we perturb the matrices as
\[
 A_i(t) = A_i\ (i=0, 1,4), A_2(t) =
 \left(\begin{smallmatrix}
	  0 & 0 & 0\\
  	  0 & 0 & 1\\
	  0 & 1 & t	
  \end{smallmatrix}\right),\ 
 A_3(t) = \left(\begin{smallmatrix}
	  1 & 0 & 0\\
  	  0 & 0 & 1\\
	  0 & 1 & t	
 \end{smallmatrix}\right),
\]
  the minimal face of each perturbed problem is equal to
  $F_{\min}$ in Example \ref{ex:notface}.  
  Here the condition on the face and the rank condition are satisfied for
  sufficiently small $t>0$. Thus Theorem \ref{thm:main2} guarantees the continuity of 
  the optimal value.
  In fact, the optimal value of \eqref{Dt} is $\frac{ 2t+4 -  4\,\sqrt{t+1}}{t^2}$
  and converges to $\frac{1}{2}$ as $t\to 0$. 
  The optimal pairs are %
{\small %
  \begin{align*}
 & X = \left(\begin{smallmatrix}
	  0&0&0\cr 0&2&{{2\,\sqrt{t+1}-2}\over{t}}\cr 0&{{2\,\sqrt{t
 +1}-2}\over{t}}&{{ 2t+4 -  4\,\sqrt{t+1}}\over{t^2}}\cr
  \end{smallmatrix}\right),\
  y = \textstyle{\left(-\frac{\sqrt{t+1}\,\left(t+2\right)-2\,t-2}{t^3+t^2}, 
 \beta, -\frac{\sqrt{t+1}+\beta\,t^2+\left(
 \beta-1\right)\,t-1}{t^2+t}, 0
  \right)},\\ 
& Z = \left(\begin{smallmatrix}
	   {{\sqrt{t+1}+ \beta\,t^2+\left(\beta-1
 \right)\,t-1}\over{t^2+t}}&0&0\cr 0&{{\sqrt{t+1}\,\left(t+2\right)-2
 \,t-2}\over{t^3+t^2}}&{{\sqrt{t+1}-t-1}\over{t^2+t}}\cr 0&{{\sqrt{t+
 1}-t-1}\over{t^2+t}}&{{1}\over{\sqrt{t+1}}}\cr
	  \end{smallmatrix}\right)
\end{align*}
} %
   for all $\beta$ such that $(1,1)$st element of $Z$ is nonnegative.
 \end{example}

 \begin{proof}[Proof of Theorem \ref{thm:main2}.]
 By the assumptions \ref{thm:sing:cond1} and \ref{thm:sing:cond2} in Theorem \ref{thm:main2},
 the optimal value of \eqref{Dt} is equal to
\[
  \inf_{X}\left\{ (Q(t)^TA_0(t)Q(t))_3 \bullet X : \right.\\
  \left. (Q(t)^TA_k(t)Q(t))_3 \bullet X = b_k(t)\ (k\in [m]), 
  X \in\Bbb{S}^r_+
 \right\}\label{FDt}\tag{\text{$F(D_t)$}},
\]
 and \ref{FDt} has a nonempty feasible set for each $t\in [0,\delta]$.
 Thus if continuity of
 the optimal value of \ref{FDt} at $t = 0$ is shown,
 then that of the optimal value of \eqref{Dt} is also shown.
 For each $t\in [0,\delta]$, we have that the dual of \ref{FDt} is
 \begin{align}
 & \sup_{y, Z}\left\{b(t)^Ty : 
  \displaystyle\sum_{k\in [m]} y_k(Q(t)^TA_k(t)Q(t))_3 + Z =  (Q(t)^TA_0Q(t))_3, 
  Z\in\Bbb{S}^r_+
 \right\}. \label{FDtd}\tag{\text{$F(D_t)^{\prime}$}}
\end{align}
 Then \ref{FDt} has the same optimal value as \ref{FDtd} 
 because \ref{FDt} and \ref{FDtd} are
 strictly feasible. 
 In fact, strict feasibility of \ref{FDt} follows from the properties of facial reduction. 
Since, for a strictly feasible point $(\tilde{y}, \tilde{Z})$ of \eqref{Pt}, $(\tilde{y}, (Q(t)^T\tilde{Z}Q(t))_3)$ is also a strictly feasible point of \ref{FDtd}, and hence \ref{FDtd} is strictly feasible.
 Therefore, the proof is done by showing Theorem \ref{thm:contopt_rank}. \hspace{\fill} \qed
 \end{proof}

 \begin{theorem}
  \label{thm:contopt_rank}
  If both  \eqref{P} and \eqref{D} are strictly feasible, 
  \eqref{Dt} is feasible, and
 $\rank\left(A_1(t),\ldots, A_m(t)\right)
 =  \rank \left( A_1,\ldots, A_m \right)$ for each sufficiently
  small $t>0$,
  then the optimal value of \eqref{Dt} varies continuously at $t=0$.
 \end{theorem}
 
  We will prove Theorem \ref{thm:contopt_rank} in Subsection
 \ref{subsection:Stab}.  
 \begin{remark}
  The coefficient matrices $A_1 ,\ldots, A_m$ in \eqref{P} are usually assumed to be linearly independent in the literature.
  However the coefficient matrices in 
  \ref{FD2}
  can be linearly dependent even if the initial SDP has linearly independent constraints.  In fact, the coefficient matrices of the reduced SDPs are linearly dependent in Examples \ref{ex:notrank}, \ref{ex:notface} and \ref{ex:theorem}. Thus we need to consider SDPs with linearly dependent coefficient matrices in Theorem \ref{thm:contopt_rank}. 
 \end{remark}
 
 As in Example \ref{ex:notrank} and \ref{ex:notface}, 
 if $\rank\left(A_1(t),\ldots, A_m(t)\right)=\rank \left( A_1,\ldots, A_m \right)$ and \eqref{D} is weakly feasible,
 then the optimal value of \eqref{Dt} can vary
 discontinuously.
 We present an additional example and show that
 the feasibility condition on \eqref{Dt} or the rank condition
 can not be removed from Theorem \ref{thm:contopt_rank}.
 
  \begin{example}
   \label{ex:perturb_infeasible}
  In \eqref{P} and \eqref{D}, we set $b=(2, 2)^T$, 
  \[
   A_0 = \begin{pmatrix}
	  0 & 0 \\
	  0 & 1
  \end{pmatrix},\
  A_1 = \begin{pmatrix}
	 1 & 0 \\
	 0 & 1
	\end{pmatrix},
  A_2 = \begin{pmatrix}
	 1 & 0 \\
	 0 & 1 
	\end{pmatrix}.
  \]
  Then \eqref{P} and \eqref{D} are strictly feasible.
   The optimal value is $0$, and the optimal pairs are $X = \left(\begin{smallmatrix}
				     2 & 0 \\
				     0 & 0
   \end{smallmatrix}\right),\
   y = \left(\alpha, -\alpha\right),\
   Z = \left(\begin{smallmatrix}
	      0 & 0 \\
	      0 & 1
	     \end{smallmatrix}\right)$ for any $\alpha \in \R$.
  However, if we take 
  $A_2(t) = \left(\begin{smallmatrix}
	  1 + t & 0\\
	  0 & 1 + t 
  \end{smallmatrix}\right)$,
   then $r(A_1(t), A_2(t)) = r(A_1, A_2)
  = 1$ but
   \eqref{Dt} is infeasible.
   Therefore feasibility of \eqref{Dt} can not be derived from the rank
   condition and needs to be assumed.

   On the other hand, if we take
   $A_2(t) = \left(\begin{smallmatrix}
		    1 + t & 0 \\
		    0 & 1 - t
   \end{smallmatrix}\right)$, then
   \eqref{Dt} is feasible and
   $r(A_1(t), A_2(t)) = 2$ for all $t>0$.
   The optimal value is $1$, and the optimal pair
   is $X = \left(\begin{smallmatrix}
		  1 & \beta \\
		  \beta & 1
   \end{smallmatrix}\right)\ (-1\leq \beta \leq 1),\
   y = \left(\frac{1+t}{2t},-\frac{1}{2t}\right),\
   Z = \left(\begin{smallmatrix}
	      0 & 0 \\
	      0 & 0 
	     \end{smallmatrix}\right)
   $. Thus the optimal value varies discontinuously at $t=0$ without the rank condition.
  \end{example}

\subsection{\textbf{Proof of Theorem \ref{thm:contopt_rank}}}    
\label{subsection:Stab}



%
%
First, we recall an existence theorem for optimal solutions to an SDP with a focus on the linear independence of the coefficient matrices.

 \begin{theorem}
  \cite[Theorem 4.1 and Corollary 4.1]{Todd01}
  \label{thm:Todd01}
 Suppose \eqref{P} is strictly feasible and \eqref{D} is feasible.
  Then \eqref{D} has a nonempty compact optimal set
  and the same optimal value as \eqref{P}.
Also, suppose that \eqref{P} is feasible and \eqref{D} is strictly
  feasible. If the coefficient matrices $A_1, \ldots, A_m$
  are linearly independent, then \eqref{P} has a nonempty compact optimal set and the same optimal value as \eqref{D}. 
 \end{theorem}
    \begin{remark}
     \label{remark:Todd01}
     \begin{enumerate}
    \item Suppose that    
    \eqref{P} is feasible and \eqref{D} is strictly feasible.
    However, we do not assume that the coefficient matrices $A_1,\ldots, A_m$
    are linearly independent. 
    Then easy arguments show that \eqref{P} has a nonempty optimal set and the same optimal value as \eqref{D}. 
    Here we lost the compactness of the optimal set of \eqref{P}.
    
   \item The set of the optimal
   solutions $(y, Z)$ of \eqref{P} is unbounded
   when the matrices $A_1, \ldots, A_m$ are linearly dependent.
   However Lemma
   \ref{lemma:unif_bdd} bellow tells that the image of the optimal solutions under the projection $(y,Z)\mapsto Z$ is bounded if \eqref{P} and \eqref{D} are strictly feasible.
   \end{enumerate}
  \end{remark}

We note that we do not assume the linear independence of the coefficient
matrices $A_1,\ldots, A_m$ in the following arguments.
We will use the symbol $S(t) =(\vec\left(A_1(t)\right), \ldots,
\vec\left(A_m(t)\right))\in\Bbb{R}^{n^2\times m}$ and the symbol $(S(t)^T)^\dagger$
for the Moor-Penrose generalized inverse of $S(t)^T$ \cite{HornJohnson13}. 
 \begin{lemma}
  \label{thm:contP}
 Suppose $X_0$ is a strictly feasible point of \eqref{D}.
 If \eqref{Dt} is feasible and\\
 $\rank\left(A_1(t),\ldots, A_m(t)\right)
 =  \rank\left(A_1,\ldots, A_m\right)$ for each $t\in [0,\delta]$, 
  then there exist strictly feasible points $X_t$ of
  \eqref{Dt} for all sufficiently small $t>0$
  such that $X_t \to X_0$ as $t \to 0$.
 \end{lemma}
   \begin{proof}
We can write the equality constraints of \eqref{D} and \eqref{Dt} by  $S(0)^T\vec(X) = b$ and $S(t)^T\vec(X) = b(t)$, respectively. 
Note that $A_k(0) = A_k\ (k\in [m]),\ b = b(0)$. We set 
\begin{align*}
\vec(X_0) &= (I - (S(0)^T)^\dagger S(0)^T)\vec(X_0) + (S(0)^T)^\dagger b(0) \mbox{ and } \\\vec(X_t) &=
   (I - (S(t)^T)^\dagger
   S(t)^T)\vec(X_0)
   + (S(t)^T)^\dagger b(t). 
\end{align*}
Then we can check that $S(0)^T\vec(X_0) = b$ and $S(t)^T\vec(X_t) = b(t)$, by using the fact that $S(t)^T(S(t)^T)^\dagger v = v$ if and only if $v\in \Im S(t)^T$. Since we have $\mbox{rank}(S(t)) = \rank (A_1(t), \ldots, A_m(t))$ for all $t\ge 0$, it follows from the assumption on the  rank and \cite[Theorem
   5.2]{Stewart69} that $(S(t)^T)^\dagger \to (S(0)^T)^\dagger$ as $t\to 0$. Therefore $X_t\to X_0$ as $t\to 0$. \hspace{\fill} \qed
   \end{proof}

 \begin{remark}
  \label{remark:s_feasible_P}
  Unlike \eqref{Dt},
  we can easily prove that \eqref{Pt} have strictly feasible points $(y_t,
  Z_t)$ for all sufficiently small $t\geq0$ without assuming the rank condition.
 If \eqref{P} is strictly feasible, there exists $y_0\in \R^m$ such that $A_0 - \sum_k y_{0,k} A_k \in \Bbb{S}_{++}^n$. Then we have that $Z_t:=A_0(t) - \sum_k y_{0,k} A_k(t)\in\Bbb{S}_{++}^n$ for all sufficiently small $t\geq0$. For each $t>0$, $(y_0, Z_t)$ is a strictly feasible point of \eqref{Pt} and converges to a strict feasible point of \eqref{P}. 
 \end{remark}


Let $\cU(t)$ be the set of optimal solutions of \eqref{Dt}, and
\[
 \cV(t) = \{Z\in \S^n : (y,Z) \text{ is optimal to \eqref{Pt}}
 \ \text{for some $y\in \R^m$}\}.
\]
  \begin{lemma}
   \label{lemma:unif_bdd}
   Suppose that \eqref{P} is strictly feasible.
   If there exist strictly feasible points $X_t$ of \eqref{Dt}
   for all sufficiently small
   $t\geq0$ such that $X_t \to X_0$ as $t \to 0$,
   then both sets $\cU(t)$ and $\cV(t)$ are nonempty and uniformly bounded;
  i.e., there exist $\delta>0$ and  compact sets $C_1, C_2$ such that
 \[
  \cU(t) \subset C_1,\ \cV(t) \subset C_2\quad (0 \leq t \leq \delta).
 \]
  \end{lemma}
 \begin{proof}
  Since \eqref{Dt} and \eqref{Pt} have strictly feasible points, Remark 
  \ref{remark:Todd01} ensures that they have the same optimal
  value and that $\cU(t)$ and $\cV(t)$ are nonempty for all sufficiently small $t\ge 0$. 
  For a strictly feasible point $(y_0, Z_0)$ of \eqref{P},
  we  set
  $  y_t = y_0$ and $Z_t = A_0(t) - \sum_ky_{0,k}A_k(t)$. 
  Then $(y_t, Z_t)$ is a strictly feasible point of
  \eqref{Pt} for each small $t\geq 0$ as explained in Remark \ref{remark:s_feasible_P}.
  Let $X$ and $(y, Z)$ be arbitrary optimal solutions to \eqref{Dt} and
  \eqref{Pt} respectively.
 Since $X_t$ and $(y_t, Z_t)$ are feasible points, we have
 \[
 A_k(t) \bullet (X - X_t) = 0,\ \sum_{k \in [m]} (y_k - y_{t,k}) A_k(t)
  + Z - Z_t = 0.
 \]
 Then it follows that 
  $(X - X_t) \bullet (Z - Z_t) = 0$ 
 and hence that
  $X \bullet Z_t + X_t \bullet Z = X_t \bullet Z_t$. 
 Moreover, positive semidefiniteness of $X_t$ and $Z$ guarantees that 
  $X \bullet Z_t \leq X_t \bullet Z_t$. 
 Thus, by positive definiteness
  of $Z_t$,  there exists $\epsilon>0$
 such that for all sufficiently small $t>0$, we have
 \[
 \|X\| \leq \frac{X_t \bullet Z_t}{\lambda_{\min}(Z_t)}
 < \frac{X_0 \bullet Z_0 + \epsilon}{\lambda_{\min}(Z_0) - \epsilon},
 \]
  where $\lambda_{\min}(M)$ is the smallest eigenvalue of a
  matrix $M$.
 Therefore, $\cU(t)$ is uniformly bounded for all sufficiently small $t>0$.
 Similar arguments are applied to $\cV(t)$. \hspace{\fill} \qed
 \end{proof}

The following lemma is well-known, and the proof is omitted.
 \begin{lemma}
  \label{lemma:minimax}
 Suppose that \eqref{D} has the same optimal value as \eqref{P} and that both
 of \eqref{D} and \eqref{P} have optimal solutions. We define the function $L :
  \Bbb{S}^n\times\Bbb{R}^m \to \Bbb{R}$ as follows: 
 \[
  L(X, y) = A_0 \bullet X + \sum_{k\in [m]} y_k(b_k - A_k
 \bullet X). 
 \]
  Then $\tX$ and $(\ty, A_0 - \sum_k \ty_kA_k)$ are optimal solutions of
  \eqref{D} and \eqref{P} respectively  if and only if $(\tX, \ty)\in\Bbb{S}^n_+\times \Bbb{R}^m$  satisfies 
 \[
  L(\tX, y) \leq L(\tX, \ty) \leq L(X, \ty),\ \forall (X,
 y) \in \S^n_+ \times \R^m.
 \]
 \end{lemma}
 \begin{lemma}
  \label{lemma:pinverse}
 Let $S$ be a matrix $\left(\vec(A_1)\ \ldots\ \vec(A_m)\right) \in \R^{n^2\times m}$. 
 If $(\ty,\tZ)$ is an optimal solution to \eqref{P},
 then $(y_*, \tZ)$ is also an optimal solution to \eqref{P}, where $y_* = S^\dagger (\vec(A_0) - \vec(\tZ))$.
 \end{lemma}
   \begin{proof}
   By feasibility of $(\ty, \tZ)$, we have $S \ty = \vec(A_0) - \vec(\tZ)$.
   Since $SS^\dagger v  = v$ if and only if  $v \in \Im S$,
   we see that $Sy_* = \vec(A_0) - \vec(\tZ)$.
   Then we obtain $y_* \in \ty + \ker S$.
    Here we have $\ker S \subset (\Span\{b\})^\perp$
    since otherwise the optimal value of \eqref{P} is infinity and hence this contradicts finiteness of the optimal value.
    Thus $b^Ty_* = b^T\ty$, and therefore, $(y_*, \tZ)$ is optimal. \hspace{\fill} \qed
   \end{proof}

Lemma \ref{thm:stability} plays an essential role in the proof of
Theorem \ref{thm:contopt_rank}. Lemma \ref{lemma:unif_bdd} and \ref{thm:stability} ensure outer semicontinuity of the set-valued map $t\mapsto \cU(t)\times \cV(t)$; see \cite[Section 5.B]{RockWets98}. 
In the following,  $\B$ denotes the closed unit ball in $\S^n$. We define, for $X\in \S^n$ and $C\subset \S^n$, 
\[
 d(X,C) = \inf\{\|X - Y\|:Y \in C\}.
\]

  \begin{lemma}
   \label{thm:stability}
   Suppose that \eqref{P} is strictly feasible.
   If there exist strictly feasible points $X_t$ of \eqref{Dt}
   for all sufficiently small $t \geq 0$ such that $X_t \to X_0$ as $t\to 0$, then 
 for any $\epsilon>0$, there exists $\eta>0$ such that
 \[
 \cU(t) \subset \cU(0) + \epsilon \B,\ \cV(t) \subset \cV(0) + \epsilon \B
 \quad (0 \leq t \leq \eta).
 \]
  \end{lemma}
 \begin{proof}
By Remark \ref{remark:Todd01},
 \eqref{Dt} and \eqref{Pt} have optimal solutions and the
  same optimal value.  Suppose that the conclusion is false.
 Then there exist $\epsilon>0$, $\{t_j\}$, $X(t_j)\in
  \cU(t_j)$ and $Z(t_j)\in \cV(t_j)$ such that
  $t_j \to 0$ and
\begin{equation}
   d\left(X(t_j), \cU(0)\right) \geq \epsilon,\ d\left(Z(t_j), \cV(0)\right) \geq
    \epsilon,
      \label{eq:contra}
\end{equation}
  for all $j$. Recall that 
  $S(t)$ denotes the matrix $\left(
	 \vec(A_1(t)) \cdots \vec(A_m(t))
	 \right)$. Let
  $y(t_j) = S(t_j)^\dagger (\vec(A_0) - \vec(Z(t_j)))$.
  Then, Lemma \ref{lemma:pinverse} implies that
  the feasible solution $(y(t_j), Z(t_j))$ is optimal for $(P_{t_j})$ for each $j$.
  We define 
  \[
  L(X, y,t) = A_0(t) \bullet X + \sum_{k\in [m]} y_k(b_k(t) - A_k(t)
  \bullet X).
  \]
  Then, we note that  $L(X, y, 0)$ is equal to $L(X, y)$ defined in Lemma \ref{lemma:minimax}. 
  By Lemma \ref{lemma:minimax}, 
 we have
  \[
  L(X(t_j), y, t_j) \leq L(X(t_j), y(t_j), t_j) \leq L(X, y(t_j), t_j),\ \forall (X,y) \in \S_+^n \times \R^m.
 \]
 Since Lemma \ref{lemma:unif_bdd} ensures that $\{(X(t_j), Z(t_j))\}$ is uniformly bounded,
  we may assume that
  \[
  (X(t_j), y(t_j), Z(t_j)) \to (\tX, \ty, \tZ)
  \]
  as $j \to \infty$ for some $(\tX, \ty, \tZ)$.
 Thus we have
\[
   L(\tX, y, 0) \leq L(\tX,\ty, 0) \leq L(X, \ty, 0),\ \forall (X,y) \in \S_+^n \times \R^m.
\] 
By applying Lemma \ref{lemma:minimax} again, $\tX$ and $(\ty, \tZ)$ are
 optimal for  \eqref{P} and \eqref{D} respectively. 
This contradicts  the inequalities \eqref{eq:contra}. \hspace{\fill} \qed
 \end{proof}

  \begin{proof}[Proof of Theorem \ref{thm:contopt_rank}.]
  By Lemma \ref{thm:contP} and \ref{thm:stability},
  we have that for any $\epsilon>0$
  and $X(t) \in \cU(t)$, there
  exist $\eta>0$ and $\tX^t \in \cU(0)$ such
  that for $t\in [0, \eta]$,
  \[
  |A_0(t) \bullet X(t) - A_0 \bullet \tX^t|
  \leq k_1\|X(t) - \tX^t\| + k_2\|A_0(t) - A_0(0)\|< \epsilon
  \]
 for some $k_1, k_2>0$. This completes the proof of Theorem \ref{thm:contopt_rank}. \hspace{\fill} \qed
  \end{proof}

 \begin{corollary}
  \label{thm:contopt}
  If both  \eqref{P} and \eqref{D} are strictly feasible and
  $A_1,\ldots, A_m$ are linearly independent, then the optimal value of \eqref{Dt} varies
 continuously at $t=0$.
 \end{corollary}
\begin{proof}
 By strict feasibility and the linear independence condition, for all sufficiently small $t>0$,
 \eqref{Pt} and \eqref{Dt} are feasible,
 and the rank condition is satisfied. 
  \hspace*{\fill} \qed
\end{proof}

\section{Behavior of a Minimal Face under Perturbations}
\label{section:face_invariant}

In this section, the behavior of a minimal face under
perturbations is investigated.
In particular, we give criteria for perturbations to keep the minimal face invariant.
%
We slightly simplify the situations and consider the following perturbed problem:
\begin{align}\label{Dt4}
&\inf_X\left\{ A_0 \bullet X :  (A_k + E_k(t)) \bullet X = b_k \ (k\in [m]), X \in\Bbb{S}^n_+\right\}, \tag{\text{$D_t$}}
\end{align}
where $E_k(t) = A_k(t) - A_k$ for all $k\in [m]$. We note $E_k(t) \to 0$ as $t\to 0$ since we assume that $A_k(t)$ are continuous at $t=0$ and $A_k(0) = A_k$.
Throughout this section, we assume the following conditions:
  \begin{condition}
   \label{cond:perturb}
  \hspace{1ex}
 \begin{enumerate}
  \item\label{D1} \eqref{D} is feasible, and \eqref{P} is strictly feasible;
  \item\label{D2} $A_1, \dots, A_m$ are linearly independent;
  \item\label{D3} \eqref{Dt4} is feasible for each sufficiently small $t>0$.
 \end{enumerate}
  \end{condition}

We say that $\{\eqref{Dt4}\}_{t\geq 0}$ satisfies the \textit{rank condition} if
there exist an associated nonsingular matrix $Q$ to the minimal face of
	\eqref{D} and $\delta>0$ such that for all $t\in [0, \delta]$,
\[
	 \rank \left( (Q^T(A_1 \!+\! E_1(t))Q)_3, \ldots,
	 (Q^T(A_m\!+\!E_m(t))Q)_3 \right) \\
	 = \rank \left( (Q^TA_1Q)_3, \ldots, (Q^TA_mQ)_3
	 \right), 
\]
where the submatrix $M_3$ for $M\in \Bbb{S}^n$ is
determined by the minimal face of \eqref{D} as in \eqref{eq:partition}.
Here we note that $Q$ in the left hand side does not depend on $t$.
We start with the following lemma.

 \begin{lemma}
  \label{lemma:rank_cond}
  Let $F_{\min}$ and $F^t_{\min}$ be the minimal faces of \eqref{D} and
  \eqref{Dt4} respectively.
  Suppose $\{\eqref{Dt4}\}_{t\geq 0}$ satisfies  the rank condition.
  If there exists $\delta>0$ such that $F^t_{\min}\subset F_{\min}$ for all $t\in [0,\delta]$, 
  we have $F^t_{\min} = F_{\min}$ for all sufficiently small $t> 0$.
 \end{lemma}
\begin{proof}
 By Lemma \ref{freduction},
 the reduced problem \ref{FD2} of \eqref{D} has a strictly feasible
  point which solves 
 $
   \left(Q^T A_k Q\right)_3 \bullet X = b_k \ (k\in [m]),  X\in\Bbb{S}^r_{++}
 $ for some $r>0$,
 where $Q$ is an associated nonsingular matrix to the minimal
 face $F_{\min}$ of \eqref{D}. 
 For each $t \in [0, \delta]$, feasibility of \eqref{Dt4} and
 $F^t_{\min}\subset F_{\min}$ imply that there exists $\tilde{X}\in F_{\min}$ such that $\left(A_k + E_k(t)\right) \bullet \tilde{X} = b_k \ (k\in [m])$. It follows from the representation of $F_{\min}$ with $Q$ 
 that
 \[
  \left(Q^T \left(A_k + E_k(t)\right) Q\right)_3 \bullet X = b_k  \
 (k\in [m]),\
 X\in\Bbb{S}^r_{+}
 \]
 is feasible. Consider the following problem obtained by perturbing \ref{FD2}:
 \begin{align}\label{pertFD}
 \inf_{X}\left\{\left(Q^TA_0Q\right)_3 \bullet X: 
 \left(Q^T \left(A_k + E_k(t)\right) Q\right)_3 \bullet X = b_k  \
 (k\in [m]),\
 X\in\Bbb{S}^r_{+}
 \right\}. 
 \end{align}
 Here, \ref{FD2} has a strictly feasible point, \eqref{pertFD} is feasible, and the rank condition is satisfied. Thus Lemma \ref{thm:contP} implies
 that for each sufficiently small $t>0$, \eqref{pertFD} has a strictly feasible point. It means that $\{X \in \rint F_{\min}: (A_k + E_k(t))\bullet X = b_k\ (k\in [m])\}\neq \emptyset$
 for each sufficiently small $t>0$. 
 Since $F_{\min}$ is a face of $\Bbb{S}^n_+$
  containing $F_{\min}^t$, we have $F_{\min}=F_{\min}^t$
 by Lemma \ref{freduction}. \hspace{\fill} \qed
\end{proof}

%

  \begin{example}   
  Lemma \ref{lemma:rank_cond} does not hold without the assumption
   $F^t_{\min} \subset F_{\min}$.
   In Example \ref{ex:notface}, 
   the perturbation  is
   of the same type as this section is considering.
   Condition \ref{cond:perturb} and the rank condition are satisfied,   
   but the minimal faces $F^t_{\min}$ of \eqref{Dt4} are not equal to
   $F_{\min}$.
   Here $F^t_{\min}$ are not included in $F_{\min}$.
  \end{example}

We first give simple sufficient conditions that can be shown easily.
 \begin{proposition}
  \label{prop:invariant1}
  For a facial reduction sequence $\{(\hy^i, \hU^i, \hV^i)\}_{i=1}^s$
  of \eqref{D},
  let the minimal face of \eqref{D} be $F_{\min}$
  and $\hat{K} = \{k: \hy_k^i = 0 \  (\forall i=1, \ldots, s)\}$.
  Suppose that $\{\eqref{Dt4}\}_{t\geq 0}$ satisfies  the rank
  condition
  and $E_k(t) = O_{n\times n} \ (k \notin \hat{K})$. 
  Then the minimal faces of \eqref{Dt4} are equal to $F_{\min}$
  for all sufficiently small $t>0$. 
 \end{proposition}
 \begin{proof}
 Let $\{F_i\}_{i=1}^s$ be the sequence of faces generated by the facial reduction sequence $\{(\hy^i, \hU^i, \hV^i)\}_{i=1}^s$
  of \eqref{D}.
 Since $E_k(t)=O_{n\times n} $ for all $k\not\in\hat{K}$, we have
  \[
  -\sum_{k\in [m]}\hy^i_k (A_k + E_k(t))
  = -\sum_{k\in [m]}\hy^i_k A_k = \hU^i + \hV^i
  \]
  for $i = 1, \ldots, s$.
  Thus $\{(\hy^i, \hU^i, \hV^i)\}_{i=1}^s$ is a facial reduction sequence  of \eqref{Dt4} up to the $s$-th iteration.
  It is summarized as
 \[
\eqref{Dt4}\quad \S_+^n \overset{(\hy^1, \hU^1, \hV^1)}{\longrightarrow} F_1
   \overset{(\hy^2, \hU^2, \hV^2)}{\longrightarrow} F_2
   \overset{(\hy^3, \hU^3, \hV^3)}{\longrightarrow}
   \cdots
   \overset{(\hy^s, \hU^s, \hV^s)}{\longrightarrow} F_s = F_{\min}.
\]
  Thus the minimal faces of \eqref{Dt4} are contained in $F_{\min}$.
  In addition, since $\{\eqref{Dt4}\}_{t\geq 0}$ satisfies the rank condition,
  it follows from Lemma \ref{lemma:rank_cond} that 
the minimal faces of \eqref{Dt4} are equal to $F_{\min}$ for sufficiently small $t>0$. \hspace{\fill} \qed
 \end{proof}

Next, we will use the positive eigenvectors of reducing certificates to give sufficient conditions for the minimal face to be invariant under a peturbation.
%
 \begin{proposition}
  \label{prop:invariant3}
 Let $\{(\hy^i, \hU^i, \hV^i)\}_{i=1}^s$ be a facial reduction sequence of
  \eqref{D},
  $F_0 = \S_+^n$ and $F_1,\ldots, F_s$ be the generated faces. 
  In addition, let
  \[
   L_i = \Span\{qq^T \colon q \text{ is an eigenvector of $\hU^i$} \\
   \text{associated with a positive eigenvalue}\}.
     \]
 Suppose that $\{\eqref{Dt4}\}_{t\geq 0}$ satisfies the rank condition, and
  for each $i = 1, \ldots, s$, 
 \[
  \sum_{k\in [m]} \hy_k^i E_k(t) +  v^i(t)
  \in
  L_i
 \]
  for some $v^i(t) \in F_{i-1}^\perp$ with $v^i(t) \to O_{n\times n}$ as $t\to 0$.
 Then \eqref{Dt4} have the same minimal face as \eqref{D} for
  all sufficiently
 small $t>0$.  
 \end{proposition}
 Before proceeding to the proof, we present an example and a remark.
  \begin{example}
   \label{ex:invariant}
  The SDP in Example $\ref{ex:notface}$ has
  a facial reduction sequence consisting of
  only one certificate
  $(\hy, \hU, \hV)=\left((0,1,-1,0)^T,\left(\begin{smallmatrix}
						      1 & 0 & 0\\
						      0 & 0 & 0\\
						      0 & 0 & 0 
						     \end{smallmatrix}\right),
  O_{3\times 3}\right)$, and hence
  $L_1 = \Span\{\hU\}$.
   If we perturb the matrices as
   {\small
   \[
 A_1(t) = \left(
 \begin{smallmatrix}
	  3t & 4t & 5t\\
  	  4t & 1 & 0\\
	  5t & 0 & t
 \end{smallmatrix}
 \right),
  A_2(t) = \left(\begin{smallmatrix}
	  0 & 3t & 2t\\
  	  3t & 0 & 1\\
	  2t & 1 & t
	      \end{smallmatrix}
 \right),
 A_3(t) = \left(
 \begin{smallmatrix}
	  1+4t & 3t & 2t\\
  	  3t & 0 & 1\\
	  2t & 1 & t
 \end{smallmatrix}
 \right),  
 A_4(t) = \left(
 \begin{smallmatrix}
	  2t & 5t & 1 + 3t\\
  	  5t & t & -t\\
	  1+3t & -t & 0
 \end{smallmatrix}
 \right),
   \]}%
   then the corresponding $E_i(t)$ satisfy
   $\sum_{k=1}^4 \hy_k E_k(t) = -4t\hU \in L_1$.
   Thus the conditions of Proposition $\ref{prop:invariant3}$ are satisfied, and
   the minimal face is invariant under the perturbation
   for sufficiently small $t>0$.
   In fact, since the minimal face of \eqref{Dt4} is contained in $\S^n_+\cap \{\hU\}^\perp$ and 
   $\left(\begin{smallmatrix}0 & 0 & 0\\
   0 & 2-t & 1-t/2\\
   0 & 1-t/2 & 1\end{smallmatrix}\right)$ 
   is a feasible point, the minimal face of \eqref{Dt4} is equal to $F_{\min}$ in Example \ref{ex:notrank}.
   More generally, to apply Proposition \ref{prop:invariant3}, 
   it suffices that we choose $E_i(t)$ such that \eqref{Dt4} are feasible, the rank condition holds, and $E_2(t) - E_3(t) = \alpha_t\hU$ for some $\alpha_t \in \R$.     
  \end{example}
  \begin{remark}
  In particular,
  the inclusion in Proposition
  \ref{prop:invariant3} holds if we have
 \[
 -\sum_{k\in [m]} \hy_k^i E_k(t) \in
 \alpha^i(t)\hU^i + F_{i-1}^\perp,
 \]
  with $\alpha^i(t)\to 0$ as $t \to 0$ for each $i=1, \ldots, s$.
 \end{remark}  
   \begin{proof}[Proof of Proposition $\ref{prop:invariant3}$]
      Since $\{(\hy^i,\hU^i,\hV^i)\}_{i=1}^s$ is a facial reduction sequence of \eqref{D}, we have that $\hU^i\in \S_+^n$, $\hV^i
  \in F_{i-1}^\perp$ and that
  $-\sum_k\hy^i_kA_k = \hU^i + \hV^i \notin F_{i-1}^\perp$ for each $i=1,\ldots, s$.
  
  Let us fix $i$. 
  Let $\{q_l\}$ be the set of the eigenvectors of $\hU^i$
  that are associated with positive eigenvalues, orthogonal to
  each other, and $\|q_l\| = 1$. Then every matrix in $L_i$ can be written
  as a linear combination of $q_l{q_l}^T$.
  By the assumption, 
  there exist
  $\alpha_{l}(t)\in \R$ and $v(t)\in F_{i-1}^\perp$
  such that $-\sum_k \hy_k^i E_k(t) = \sum_l\alpha_{l}(t)q_l{q_l}^T +
    v(t)$ and $v(t)\to O_{n\times n}$.
  Since $\sum_k \hy_k^i E_k(t) \to O_{n
    \times n}$ 
    as $t\to 0$
   and $\{q_lq_l^T\}$ is linearly independent,    
  we have $\alpha_{l}(t)\to 0$ for each $l$. 
  We set
  \[
   U^i = \hU^i + \sum_l\alpha_{l}(t)q_l{q_l}^T,\quad 
  V^i = \hV^i + v(t)
  \]
    Then $V^i \in F_{i-1}^\perp$.
    Since $\hU^i$ can be written as $\sum_l\lambda_lq_l{q_l}^T$,
    where $\lambda_l$ is the positive eigenvalue of $\hU^i$
    corresponding to $q_l$, 
  we see that $U^i\in \S_+^n$ for all sufficiently small $t>0$. 
  Thus we have
  \[
  -\sum_k \hy_k^i \left( A_k + E_k(t) \right) = \hU^i + \hV^i + \sum_\ell \alpha_\ell(t)q_\ell q_\ell^T + v(t) = U^i + V^i.
  \]
  Since $\hU^i + \hV^i \notin F_{i-1}^\perp$ by the definition of the facial reduction sequence and
  $F_{i-1}^\perp$ is closed,
  we also have $U^i + V^i \notin F_{i-1}^\perp$ for all
  sufficiently small $t>0$.
  In addition, we obtain that
\begin{align*}
  F_{i-1} \cap \left\{U^i\right\}^\perp
 & = F_{i-1} \cap \Big\{ \hU^i + \sum_l\alpha_{l}(t)
 q_{l}q_{l}^T\Big\}^\perp\\
 &  = F_{i-1} \cap \Big\{ \sum_l(\lambda_l + \alpha_{l}(t)
 )q_{l}q_{l}^T\Big\}^\perp
  = F_{i-1} \cap \Big\{\hU^i\Big\}^\perp = F_i.
\end{align*}

Therefore, we have shown that $\{U^i\}_{i=1}^s$ also generates
  the faces $F_1,\dots,F_s$ 
  and that $(\hy^i,U^i,V^i)$ is a reducing certificate of \eqref{Dt4} at the $i$-th iteration for each $i=1, \ldots, s$.
  Thus $F_s$ contains
  the minimal face of \eqref{Dt4} for each sufficiently small $t>0$.
 In addition, since $\{\eqref{Dt4}\}_{t\geq 0}$ satisfies the rank condition,
    Lemma
    \ref{lemma:rank_cond} implies that
    the minimal face of \eqref{Dt4} is equal to $F_s$ for each sufficiently small $t>0$. \hspace*{\fill} \qed
   \end{proof}


\section{Application to a Control Problem}\label{sec:example}

\subsection{\textbf{A Singular SDP in $H_\infty$ State Feedback Control Problem}}\label{subsec:example}
We present a singular SDP arising from $H_\infty$ state feedback control problem. The $H_{\infty}$ control problem is one of the most successful applications
of SDP and is the problem for designing a controller that  achieves stabilization with some guaranteed performance based on the $H_\infty$ norm. In particular, the $H_{\infty}$ state feedback control problem is a special case of the $H_{\infty}$ control problem. See, e.g., \cite{Iwasaki94,Scherer06} for the detail on the SDP formulation. 

In this section, we deal with the following SDP problem: 
\begin{align}
\label{LMI}
&\sup\left\{ -\gamma: \begin{pmatrix}
   -\He(AY_1 + B_2Y_2) &  &  \\
   -C_1Y_1 - D_{12}Y_2       & \gamma I_2 &  \\
   -B_1^T                & -D_{11}^T & \gamma I_{2}
 \end{pmatrix}\in \S_+^6, \begin{array}{l}
 Y_1\in \S_+^2, \\
 Y_2\in \R^{2\times 2},\\
  \gamma\in\Bbb{R}
  \end{array}\right\},\tag{P0}
\end{align}
where $\He(X) = X+ X^T$ for $X\in\Bbb{R}^{n\times n}$ and the blanks in the matrices stand for the transpose of the lower triangular block part. Also,  the matrices $A$, $B_1$, $B_2$, $C_{1}$, $D_{11}$ and $D_{12}$ are defined  as follows:
\begin{equation}
 \left(
 \begin{array}{c|c|c}
  A & B_1 & B_2 \\
  \hline
  C_1 & D_{11} & D_{12}
 \end{array}\right)
 = \left(
\begin{array}{cc|cc|c}
 -1 & -1 & -1 & -1 & 0 \\
 1  &  0 & -1 &  0 & 1 \\
 \hline
  2  & -1 & -1 & 0  & 2 \\
 -1  & 2  & -1 & 0  & -1
\end{array}
 \right).  
\label{eq:control_matrix}
\end{equation}
Its dual can be formulated as follows:
\begin{align}
\label{DUAL}
&\inf\left\{
-\begin{pmatrix}
O&&\\
O&O&\\
B_1^T&D_{11}^T &O 
\end{pmatrix} \bullet Z
: \begin{array}{l}
\He(A^TZ_{11}+C_1^TZ_{21})\in\mathbb{S}^{2}_+, \\
I_p\bullet Z_{22} + I_{m_1}\bullet Z_{33} = 1, \\
B_2^TZ_{11}+D_{12}^TZ_{21}=O, \\
Z=\left(Z_{ij}\right)_{1\le i, j\le 3}\in\mathbb{S}^{
6}_+
\end{array} 
\right\}. 
\end{align}
%
%
To adjust our SDP problem of interest to the form of \eqref{P}, we define the coefficient matrices $A_k$ \ $(k\in [6]\cup\{0\})$ and vector $b$ by 
$
A_k = \left(\begin{smallmatrix}
A_{k, 1} & O\\
O & A_{k, 2}
\end{smallmatrix}\right)$ and $b = \begin{pmatrix}
0 & 0 & 0 & 0 & 0 & -1
\end{pmatrix}^T$, where $A_{k, 1}\in\mathbb{S}^6$ and $A_{k, 2}\in\mathbb{S}^2$ for all $k$. Furthermore, we rewrite variables $Y_1$, $Y_2$,  $\gamma$ as follows:
\[
Y_1 = \begin{pmatrix}
y_1 & y_2 \\
y_2 & y_3
\end{pmatrix}, Y_2 = \begin{pmatrix}
y_4 & y_5
\end{pmatrix}, y_6 = \gamma. 
\]
It follows from \cite[Theorems 3.3 and 3.5]{Waki15} that
\eqref{LMI} is strictly feasible but its dual problem is weakly feasible.
 Thus we can say that \eqref{LMI} is singular. 
 
We compare computational results on \eqref{LMI} with the following  three perturbed SDPs for \eqref{LMI}: For $\epsilon = $1.0e-16, 
\begin{enumerate}[label=(P\arabic*),leftmargin=*]
\item SDP obtained by perturbing the $(2, 2)$nd element of $A_{5, 1}$  into $-2(1+\epsilon)$,\label{p1}
\item SDP obtained by perturbing  the $(2, 3)$rd and $(3, 2)$nd elements in $A_{5, 1}$  into  $-2(1+\epsilon)$, and   \label{p2}
\item SDP obtained by perturbing the $(2, 4)$th and $(4, 2)$nd elements of $A_{5, 1}$ into $1+\epsilon$.  \label{p3}
\end{enumerate} 
We apply SDPA-GMP \cite{SDPA} to solve \eqref{LMI} to \ref{p3} with stopping tolerances $\delta$ ($\delta$=1.0e-10, 1.0e-30 and 1.0e-50) and  set the floating point computation to  approximately 300 significant digits. 
We set  {\ttfamily maxIteration} = 10000 and {\ttfamily betaStar} = {\ttfamily betaBar} = {\ttfamily gammaStar} = 0.5 for parameters of SDPA-GMP. See \cite{SDPA} for more details on parameters. Table \ref{results} shows the numerical results. We observe the following:

\begin{table}[htbp]
\caption{Computed values for  \eqref{LMI},  its perturbed problems \ref{p1}, \ref{p2} and \ref{p3}}
\centering
\begin{tabular}{|c|c|c|c|}
\hline
& $\delta=$1.0e-10 &$\delta=$1.0e-30 &$\delta=$1.0e-50 \\
\hline
\eqref{LMI} &-2.2360679775444764  &-2.2360679774997897 &-2.2360679774997897 \\
\ref{p1} &-2.2360072694172072  &-2.1078335768712432 & -1.4142135623730950\\
\ref{p2} &-2.2360072694172055  &-2.0000000000000000 & -2.0000000000000000\\
\ref{p3} &-2.2360072665294605  &-1.4142135623730950 &-1.4142135623730950\\
\hline
\end{tabular}
\label{results}
\end{table}%

\begin{itemize}
\item The computed values of \eqref{LMI} are almost same
      for all $\delta$, whereas the values for perturbed problems \ref{p1}, \ref{p2} and \ref{p3} are different.  These significant differences imply that one needs to choose suitable
      tolerances $\delta$ in order to use the floating point computation
      with longer significant digits for singular SDPs.
      
\item We can verify that the optimal value of  \eqref{LMI} is $-\sqrt{5}$, while the
      optimal value of the perturbed problem \ref{p1} is $-\sqrt{2}$. 
      These differences  show that a small perturbation of coefficient matrices $A_k$ in \eqref{LMI} may yield a  significant change of the optimal value of \eqref{LMI}. 
\end{itemize}

\subsection{\textbf{Behavior of Minimal Faces under Perturbations for Our Example}}\label{example:invariant}
We show that matrix-wise perturbations make the minimal face of the dual problem of
  \eqref{LMI} invariant or full-dimensional, i.e., $\Bbb{S}^8_+$. 

 Let $A = \left(\begin{smallmatrix}
	   a_{11} & a_{12} \\
	   a_{21} & a_{22}
      \end{smallmatrix}\right)$, 
 $B_2 = \left(\begin{smallmatrix}
	 b_1 \\
	 b_2
 \end{smallmatrix}\right)$, 
 $C_1 = \left(\begin{smallmatrix}
	 c_{11} & c_{12} \\
	 c_{21} & c_{22}
	\end{smallmatrix}\right)$, 
 $D_{12} = \left(\begin{smallmatrix}
 	    d_{11}\\
	    d_{21}
	   \end{smallmatrix}\right)$, and let $B_1$ and $D_{11}$ be the
  same matrices as in \eqref{eq:control_matrix}. Then the first 
  constraint in  \eqref{LMI} 
  means that 
{\footnotesize
 \[
 \begin{pmatrix}
  -2a_{11}y_1 - 2a_{12}y_2- 2b_1y_4 &  \\
  -a_{21}y_1 - (a_{11} + a_{22})y_2
  - a_{12} y_3 - b_2 y_4 - b_1 y_5
  & - 2a_{21} y_2 - 2a_{22} y_3 - 2b_2 y_5 & \\
  -c_{11}y_1 - c_{12}y_2 - d_1y_4 & -c_{11}y_2 - c_{12}y_3 - d_1y_5
  & y_6 \\
   -c_{21}y_1 - c_{22}y_2 - d_2y_4 & -c_{21}y_2 - c_{22}y_3 - d_2 y_5
  & 0 & y_6 \\
  1 & 1 & 1 & 1 & y_6  \\
  1 & 0 & 0 & 0 & 0 & y_6
 \end{pmatrix}
 \]
  }
  is contained in $\S_+^6$.
The related part with $a_{11}$ in the above matrix can be extracted as
 \begin{align*}
a_{11}
 \left(\begin{smallmatrix}
  -2y_1 & -y_2 & 0 & 0 & 0 & 0 \\
  -y_2 & 0 & 0 & 0 & 0 & 0 \\
  0 & 0 & 0 & 0 & 0 & 0 \\
  0 & 0 & 0 & 0 & 0 & 0 \\
  0 & 0 & 0 & 0 & 0 & 0 \\
  0 & 0 & 0 & 0 & 0 & 0
 \end{smallmatrix}\right)&=a_{11}y_1\left(\begin{smallmatrix}
  -2 &0 & 0 & 0 & 0 & 0 \\
  0 & 0 & 0 & 0 & 0 & 0 \\
  0 & 0 & 0 & 0 & 0 & 0 \\
  0 & 0 & 0 & 0 & 0 & 0 \\
  0 & 0 & 0 & 0 & 0 & 0 \\
  0 & 0 & 0 & 0 & 0 & 0
 \end{smallmatrix}\right) +a_{11}y_2 \left(\begin{smallmatrix}
  0 & -1 & 0 & 0 & 0 & 0 \\
  -1 & 0 & 0 & 0 & 0 & 0 \\
  0 & 0 & 0 & 0 & 0 & 0 \\
  0 & 0 & 0 & 0 & 0 & 0 \\
  0 & 0 & 0 & 0 & 0 & 0 \\
  0 & 0 & 0 & 0 & 0 & 0
  \end{smallmatrix}\right)\\
  &=:a_{11}(y_1 E_{1,1} + y_2 E_{2,1}). 
 \end{align*}
 Since a perturbation on $a_{11}$ affects the coefficient  matrices of $y_1$ and $y_2$, the corresponding perturbing matrices are
 $E_1(t)=\left(\begin{smallmatrix}
 tE_{1,1} & O \\
 O&O_{2\times 2}\end{smallmatrix}\right)$, $E_2(t) =\left(\begin{smallmatrix}
 tE_{2,1} & O \\
 O&O_{2\times 2}\end{smallmatrix}\right)$
  and $E_k(t) = O_{8\times 8}\ (k =
  3, \ldots, 6)$.
  We remark that we need to consider block matrices with two blocks for the perturbation because the coefficient matrices for $y_1$ also appear in the constraint $Y_1 \in \S^2_+$ of \eqref{LMI}.  
  
  Consider the problem \eqref{Dt4} perturbed with $\{E_k(t)\}$.
  Then one can verify that the length of  the facial reduction sequence
  for \eqref{Dt4} is one and that it is $\{(y, U, V)\}$, where
  \begin{align}\label{eq:ex1}
&\left\{
\begin{array}{l}
 y = (1, 0, 0, -1, 0, 0)^T, U = \begin{pmatrix}
 U_1 & O\\
 O& U_2
 \end{pmatrix}, V = \begin{pmatrix}
 V_1 & O\\
 O& V_2
 \end{pmatrix},\\
 V_1 =O_{6\times 6}, 
 V_2=O_{2\times 2}, U_1= \begin{pmatrix}
1 & 0^T\\
0 & O_{5\times 5}
\end{pmatrix}, U_2 = \begin{pmatrix}
		      1 & 0\\
		      0 & 0
		     \end{pmatrix}.
		     \end{array}
		     \right.   
\end{align}
  Let $e_1\in\Bbb{R}^6$ and $f_1\in\Bbb{R}^2$ be the unit vectors whose first entry is $1$ and others are
  zero. Then the positive eigenvalues of $U$ are $2, 1$, and
  the associated eigenvectors
  are $(e_1, 0_2^T)^T$, $(0_6^T, f_1^T)^T$ respectively. Here $0_p$ is the $p$-dimensional zero vector for a given positive integer $p$. 
  Since we have that
  \[
  -\left( 1 \cdot E_1(t) + 0 \cdot E_2(t) \right)
  \in\Span\left\{\left(\begin{smallmatrix}e_1{e_1}^T & O \\ O & O_{2\times 2}\end{smallmatrix}\right), \left(\begin{smallmatrix} O_{6\times 6} & O \\ O & f_1{f_1}^T\end{smallmatrix}\right)
  \right\}
  \]  
  and that $\{\eqref{Dt4}\}_{t\geq 0}$ satisfies the rank condition,
  Proposition \ref{prop:invariant3} implies that this perturbation does not
 change the minimal face of the dual problem.


We can apply similar arguments to see behavior of the  minimal face of the dual problem for the other perturbations and observe the followings: 
\begin{itemize}
\item The minimal face is invariant under the matrix-wise perturbation with respect to $a_{11}$, $a_{12}$, $a_{22}$, $c_{12}$, $c_{22}$ and $b_1$. The optimal value of \eqref{Dt4} changes continuously at $t = 0$ due to Theorem
3.1. 
\item The other perturbations, i.e. $a_{21}$, $c_{11}$, $c_{21}$, $b_2$ $d_1$ and $d_2$,  make the minimal face 
  of the dual problem to be $\Bbb{S}^8_+$,   which implies that the perturbed problem is strictly
  feasible.  However, we have numerically
confirmed that the optimal value of \eqref{Dt4} also varies continuously in this case. It is a future study to find other
conditions that ensure the
continuity of the optimal value under any matrix-wise perturbations.
  \item  Hence if we perturb matrices $A$, $B_2$, $C_1$ and
  $D_{12}$  in the structured form, the minimal face may be
  different,
  but can not  be smaller.
\end{itemize}


\section{Conclusions}\label{sec:conclusion}
We consider perturbations of the coefficient matrices of a semidefinite program,
in the case that the primal problem is strictly feasible and the dual problem is weakly
feasible.
We give sufficient conditions for continuity of the optimal value.
These conditions involve the behavior of the minimal faces of the perturbed dual problems
and the submatrices of the coefficient matrices associated with the minimal faces.
By using examples, it is argued that these conditions are hard to
remove.
We further obtain sufficient conditions for the perturbations to keep
the minimal face invariant.
A facial reduction sequence, which is obtained in the process of facial
reduction, plays the central role.
Then our results are applied to a semidefinite program obtained from
an $H_\infty$ control problem.
By presenting numerical experiments with interior point methods,
we also discuss the importance of computations with arbitrary precision arithmetic, together with an appropriate
parameter for the stopping criteria,
in order to obtain an approximation to the optimal value of a singular semidefinite program.

In the future work, it is worth considering to use a facial reduction sequence to
analyze other properties of a semidefinite program.
In addition, it may be interesting to find combinatorial structures
in the elements of perturbing matrices that preserve the minimal face of
a semidefinite program.

\section{Acknowledgements}

The first author was supported by  JSPS KAKENHI Grant Number JP15K04993 and JP19K03631.
A part of his work was done when he stayed in University
of Konstanz with the financial support from Tokyo University of Marine
Science and Technology.  
The second author was supported by JSPS KAKENHI Grant Numbers
JP22740056, JP26400203, JP17H01700, JP20K11696, and ERATO HASUO
Metamathematics for Systems Design Project (No.JPMJER1603), JST.
We would like to thank  Noboru
Sebe in Kyushu Institute of Technology for fruitful discussions and
significant comments for $H_{\infty}$ control
problems.

\end{document}